\documentclass[11pt, reqno]{amsart}
\usepackage{fullpage}
\usepackage{amssymb,enumerate}
\usepackage[colorlinks=true,linkcolor=blue]{hyperref}
\theoremstyle{plain}

\usepackage{minted}
\usepackage{graphicx,comment}

\hypersetup{
    colorlinks = true,
    citecolor= {teal},
    linkcolor = {teal},
    urlcolor = {teal}
}

\DeclareMathOperator{\diag}{diag}

\DeclareMathOperator{\ran}{ran}

\newtheorem{theorem}{Theorem}[section]

\newtheorem{corollary}[theorem]{Corollary}
\newtheorem{proposition}[theorem]{Proposition}
\theoremstyle{remark}

\newtheorem{remark}[theorem]{Remark}

\numberwithin{equation}{section}

\newcommand{\Comp}{\mathbb C}
\newcommand{\norm}[1]{\left\| #1\right\|}

\title[Index Detection]{A Perturbation Method for Index Detection for Linear Matrix Pencils }

\author{Hanna Blazhko}
\author{Micha\l{} Wojtylak}
\begin{document}
\begin{abstract}
Rigorous, non-asymptotic bounds for the Puiseux expansion of the eigenvalue at infinity are given.
Error analysis is provided.
Further, the expected value of the eigenvector condition number of a randomly perturbed matrix is estimated. The latter result is applied to the Cayley  transform of the linear pencil.  Numerical simulations illustrating the theoretical findings are provided. 
\end{abstract}

\maketitle

%%%%%%%%%%%%%%%%%%%%%%%%%%%%%%%%%%%%%%%%%%%%%%%%%%%%%%%%%%%%%%%%%%%%
\section{Introduction}

In recent years, significant progress has been made in randomized methods for computing eigenvalues of linear pencils, in particular for singular pencils \cite{Hochstenbach_2019,Hochstenbach_2023,Kressner_2024} and singular quadratic polynomials \cite{kressner2023singular}, thereby overcoming the classical limitation \cite{Wilkinson1979}. These methods—based on random perturbations of a prescribed rank, random projections, and random augmentations—often demonstrate performance superior to that of the classical staircase algorithm \cite{VanDooren1979}. While the new approach successfully determines the eigenvalues, it does not address the computation of the Kronecker form. The aim of the present work is to make progress in this direction. For simplicity, we restrict our attention to regular pencils, and on the \emph{index}, i.e., the size of the largest Kronecker block associated with the eigenvalue at infinity. We also briefly indicate how the results on index can be transferred to finite eigenvalues.
We refer to  \cite{emmrich2013operator,saak2018model,mehrmann2023control,freund2004extension} for  applications of index $2$ systems.
While our work is pure finite-dimensional, it is also related  to a recent infinite-dimensional study on index  concepts, studied in detail in
\cite{erbay2024index}. 

We approach the problem using the first-order perturbation theory developed  in \cite{Lidski66} and  \cite{de2008first}. The latter paper  states that 
for a pencil $\lambda E - A$ with an eigenvalue infinity with the largest  Kronecker block of size $k$, the eigenvalues of the perturbed pencil
\[
\lambda E - A + \tau (\lambda X - Y),
\]
where $\lambda X - Y \in \mathbb{C}^{n\times n}[\lambda]$ is a generic perturbation, admit, for sufficiently small $\tau$, a Puiseux expansion 
\begin{align}\label{eq:expansion}
\lambda(\tau) =  C \tau^{-1/k} + o(\tau^{-1/k})
\end{align}
with some constant $C\neq0$.
Our idea is to use this expansion to determine the index. However, recall that the numerical algorithms tend to 
replace the Kronecker blocks of size $k$ by $k$ single eigenvalues in the neighborhood of $\lambda_0$, see, e.g., \cite{ahmad2010pseudospectra,alam2005sensitivity,stewart1972sensitivity,sun1991perturbation}.
Hence, the first  question that needs to be posed for applying the method follows:\\

(Q1) \emph{ Can the fractional expansion \eqref{eq:expansion} be observed in numerical simulations? }\\

Rephrasing the question, we are interested in whether 
 a small perturbation $\Delta(\lambda)$ of $\lambda E-A$ leads to a corresponding perturbation of the behavior described by \eqref{eq:expansion}.  
We give a positive answer to this question, providing (in case $k=2$) a rigorous non-asymptotic estimate 
$C_1 \tau^{-1/2}  \leq |\lambda_\Delta(\tau)|\leq C_2\tau^{-1/2}$ of the eigenvalue $\lambda_\Delta(\tau)$ of the perturbed pencil.    
The next question we address is\\

(Q2) \emph{ Is the same behavior observed in  the discretization?}\\

For that, we consider the simplest discretization model, given by the Cayley transform of the pencil,
$$
\left( E - hA \right)^{-1}\left( E + hA \right).
$$
In our approach, we perturb this matrix with a random matrix and average the results over several samples (see below for details). Our theoretical results here are based on a recent technique \cite{banks2021gaussian} that allows controlling the eigenvector condition number of the perturbed matrix.
This allows us to give a positive  answer to the second question.
  This leads to the third, central objective:\\

(Q3) \emph{ Can the observed behavior of the perturbed eigenvalues be used to determine the index?}\\

Our theoretical and numerical results confirm that it is possible to determine whether a given pencil lies in proximity to the set of index-two pencils.
In this sense, our work can be seen as solving a distance problem, see \eqref{deltaE} for details on the norm used. 
We refer here to  \cite{akinola2014calculation,alam2005construction,alam2011characterization,giesbrecht2017computing,kressner2015distance}, and  \cite{gillis2018computing,MehMW21}  for related distance problems for matrices and  matrix pencils, respectively.

Although the theorems are formulated for general matrices, they are particularly suited for 
pencils connected with energy-based modeling. We refer the reader to \cite{mehrmann2023control} for an overview and to \cite{gillis2018computing,gillis2018finding,gillis2022solving} for related distance problems. Several applications appear  in Section~\ref{sec:examples}. 
In particular, we discuss dissipative Hamiltonian pencils, i.e., pencils of the form $\lambda E - (J - R)Q$, where all matrices are complex square ones, satisfying additionally $E^*Q \geq 0$, $R \geq 0$, and $J = -J^*$.
The Kronecker form of these pencils is described in \cite{MehMW18,MehMW21} and \cite{faulwasser2022optimal}. In Section~\ref{sec:pH}, we contribute to the topic by discussing the case $Q>0$. The remarkable property of such pencils is that all their eigenvalues lie in the closed left half-plane, and the eigenvalues on the imaginary axis are always semisimple, except possibly for a Kronecker block of size two at zero or at infinity. Notably, $R>0$ implies index one, see \cite{gillis2018computing}. On the other hand,  Kronecker blocks of size one violate minimality of the system, see \cite{freund2004extension}, which motivates the main assumption in Theorem~\ref{Theorem1} below and in subsequent corollaries.  An additional motivation for (Q2) is  that in this setting the Cayley transform is the basis for the implicit midpoint rule, one of the simplest discretization rules preserving the energy, cf.  \cite{kotyczka2021symplectic}
 for implementation  for port-Hamiltonian system.  

We discuss now the main results of the paper. We consider complex matrices with the euclidean norm. A central role in both the deterministic and probabilistic analyses is played by the \emph{eigenvector condition number} defined as 
\begin{equation*}
    \kappa_V(X):= \inf\left\{\kappa(T): T\in\Comp^{n,n},\ T \text{ invertible},\ T^{-1}XT  \text{ diagonal } \right\},
\end{equation*} 
where $\kappa(T):= \norm{T}\cdot \norm{T^{-1}}$ is the \emph{condition number} of a matrix.

%%%%%%%%%%%%%%%%%%%%%%%%%%%%%%%%%%%%%%%%%%%%%%
\subsection{Deterministic results} 
The symbol $\sigma_{\min}(E)$  denotes the smallest singular value of $E$ (possibly zero).
Throughout, we  also use
\begin{equation}\label{projections}
E_{11}:=P_{\ran E} E P_{\ran E},
\quad
A_{11}:=P_{\ran E}A P_{\ran E},
\quad
A_{12}:=P_{\ran E}A P_{\ker E},
\quad
A_{21}:=P_{\ker E}A P_{\ran E},
\end{equation}
where $P_{\mathcal X}\in\Comp^{n,n}$ stands for the orthogonal projection onto a subspace $\mathcal X$ of $\Comp^n$.

Furthermore, we define
\begin{equation}\label{eq:mu_bound} \tau_0:= \sup \left\{\tau\geq0 :|\mu(\tau)| > \frac{\norm{ A_{11}}}{\tau+\sigma_{\min}(E_{11})/2}\text{ for some ev. }\mu(\tau)\text{ of }\lambda(E+\tau I)-A \right\}. \end{equation}
Note that $\tau_0>0$, provided that infinity is an eigenvalue of $\lambda E-A$, and $\tau_0=\infty$ if, additionally, $A_{11}=0$.

\begin{theorem}\label{Theorem1} Let $\lambda E -A\in\Comp^{n,n}[\lambda]$ be a regular pencil  of index two with no Kronecker blocks corresponding to infinity of size one and with $E\geq 0$. 
Then a perturbed pencil 
\begin{equation}\label{deltaE}
    \lambda (E + \tau I  + \Delta_E) -  (A +\Delta_A), 
    \quad \norm{\Delta_E}, \norm{\Delta_A}\leq \delta<\sigma_{\min}(E_{11}),
\end{equation}
 is regular and  has an eigenvalue $\mu_\Delta(\tau)$ satisfying
\begin{equation}\label{finalsquare2}
\frac1{\sqrt\tau}\sqrt{\frac{\sigma_{\min} (A_{12}A_{21}) }{\frac32 {\norm{E_{11}}+2\tau } }} - \gamma(\delta,\tau)
\leq
|\mu_\Delta(\tau)|
\leq
 \frac1{\sqrt\tau}\sqrt{\frac{\norm{A_{12}A_{21}} } {\frac 12 \sigma_{\min} ( E_{11})} } +\gamma(\delta,\tau), 
\end{equation}
where
\begin{equation}\label{gamma}
 \gamma(\delta,\tau)= \delta\cdot \kappa_V( (E +\tau I)^{-1}A  )  \frac{\tau+\norm A}{\tau(\tau-\delta)},\quad 0\leq\delta<\tau \leq \tau_0.
\end{equation}

\end{theorem}

%---------------------------------------------
The formulation of Theorem~\ref{Theorem1} as stated above is optimal for general purposes. However, it can be refined with sharper estimates in certain special cases and extended in various ways to cover more general scenarios. In the following sections, we present several such adaptations tailored to our numerical experiments. For convenience, we summarize them here.

\begin{itemize}
    \item The case $\delta=0$ provides  bounds on the leading coefficient in the Puiseux expansion in the index 2 case.
    Meanwhile, the case $\delta > 0$ shows how the numerical results deviate from the square-root bounds when the pencil is $\delta$-close to an index-two pencil.
   
    \item Sharper estimates in the case $\Delta_E = 0$ and for dissipative Hamiltonian pencils are given in Corollary~\ref{cor:deltaE0} and Corollary~\ref{cor:pH}, respectively.

    \item Generalizations to a diagonalizable $E$ and arbitrary $E$ are given in Remark~\ref{cor:T} and Remark~\ref{rem:arbitraryE}, respectively.

    \item An extension to arbitrary finite eigenvalues of the pencil $\lambda E - A$ is discussed in Remark~\ref{rem:finite}.
\end{itemize}

A limitation of Theorem~\ref{Theorem1} is the presence of the eigenvector condition number $\kappa_V( (E+\tau I)^{-1} A )$ which may in principle be unbounded. We address this issue in two ways.
First, in Section~\ref{sec:pH}, we consider a special class of dissipative Hamiltonian pencils for which this condition number admits a sharper upper bound, see Corollary~\ref{cor:pH}. 
Second, we replace the deterministic perturbation $\tau I$ by probabilistic perturbations, allowing us to control the eigenvector condition number. This approach is discussed in more detail in the following subsection.

%------------------------------------------------
\subsection{Probabilistic results}

The second theoretical result of the paper is presented below.  By a \emph{complex Ginibre matrix} we understand an $n \times n$ random matrix $G_n = [g_{ij}]_{ij=1}^n$ with i.i.d. complex Gaussian entries, i.e. with real
and imaginary parts being independent $\mathcal N(0, 1/2n)$ random variables. Let $M$ and $\Delta$ be two arbitrary matrices.
Recall that for fixed $\tau>0$   both  matrices $M+\tau G_n$ and $M+\Delta+\tau G_n$ have $n$ distinct eigenvalues with probability one. We denote them, respectively, by $\lambda_i(\tau)$ and $\lambda_i^\Delta(\tau)$, $i=1,\dots n$, $\tau>0$. The ordering is not essential in what follows. The symbol $\mathbb E$ denotes the expected value (with respect to the Ginibre matrix). 

%------------------------------------------------
\begin{theorem}\label{thmM}
Let $M,\Delta\in\Comp^{n,n}$ be   matrices  with $\norm M,\norm{M+\Delta}\leq 1$. 
Then there exists a universal constant $\alpha(n)$, dependent only on the dimension $n$, such that 
\begin{equation}\label{probbound}
\left| \mathbb E\min_{i=1,\dots,n}|\lambda_i^{\Delta}(\tau)| - \mathbb E\min_{i=1,\dots,n}|\lambda_i(\tau)| \right|
\leq 
\frac{\alpha(n) \norm{\Delta}}{\tau} ,\quad \tau\in (0,1).
\end{equation}

Moreover, $\lim_{n\to\infty}\frac{\alpha(n)}{n^{3/2}} < 7.82843$.
\end{theorem}

In the subsequent sections, we apply this result to $M=\frac{M_h(E,A)+I}{\norm  {M_h(E,A)+I }   }$, where 
\begin{align*}
   M_h(E,A) := \left[ E - h A \right]^{-1} \left[ E + h A  \right] 
\end{align*}
for such $h\in\mathbb R$ that matrix $ (E - h A )$ is invertible. 
Recall that if $\lambda_0$ is an eigenvalue of the pencil $\lambda E-A$ then $\mu_0=\frac{1+h\lambda_0}{1- h \lambda_0} $ is an eigenvalue of $M_h(E,A)$, and the structure of the Kronecker blocks of $\lambda E-A$ at $\lambda_0$ and the Jordan blocks of $M_h(E,A)$ at $\mu_0$ coincide, see \cite{mackey2015mobius}. In particular, if $\lambda E-A$ has an eigenvalue $\infty$ of index 2 then, for $\tau$ in a neighborhood of 0, the perturbed matrix $M+\tau G_n$ has eigenvalues with expansion
\begin{align}\label{eq:le_mu}
\lambda(\tau) =  C'\tau^{1/2} + o(\tau ^{1/2})    
\end{align}
for some constant $C'$, depending on the sample of the random matrix.  
Theorem~\ref{thmM} estimates the deviation of this bound when perturbing $M$ by a small norm matrix $\Delta$.

%------------------------------------------------
\subsection{Organization of the paper}

Two main theorems have already been formulated above (Theorem~\ref{Theorem1} and Theorem~\ref{thmM}). We continue with a detailed analysis of Theorem~\ref{Theorem1} in Section~\ref{sec:thm1}, therein one may find the full proof and corollaries mentioned above. 
In Section~\ref{sec:pH}, we discuss the dissipative Hamiltonian setting. Section~\ref{sec:thm2} is devoted to probabilistic reasoning constituting the proof of Theorem~\ref{thmM}. 
Finally, in Section~\ref{sec:examples}, we focus on practical aspects and compare both approaches through several simulations motivated by real-life examples.

%%%%%%%%%%%%%%%%%%%%%%%%%%%%%%%%%%%%%%%%%%%%%%
\section{Theorem~\ref{Theorem1}: auxiliary results, proof, and corollaries}\label{sec:thm1}
First, we establish an auxiliary result.

\begin{proposition}\label{prop:index2_form_RJ}
    Let $\lambda E-A$ be a pencil with $E\geq0$. 
    Then the following are equivalent:
    \begin{enumerate}[\rm (i)]
        \item  $\lambda E-A$ is regular of index at least two with no Kronecker blocks corresponding to infinity of size one;
        \item for some unitary matrix $U$ one has  $U^* (\lambda E-A)U$  of the form 
    \begin{equation}\label{EAblock}
\lambda \begin{bmatrix} E_{11}&0\\ 0&0 \end{bmatrix} -  \begin{bmatrix}
    A_{11}&A_{12}\\ A_{21}&0
\end{bmatrix},
\end{equation}
    with $A_{12}$ having full column rank, $A_{21}$ having full row rank, $E_{11}$ being diagonal and invertible, and the size of $E_{11}$ being strictly less than the size of $E$.
    \end{enumerate}
Furthermore, if $A=J-R$ with $J=-J^*$ and $R\geq 0$ then $U^*JU$ and $U^*RU$ can be partitioned conformably with \eqref{EAblock} as
\begin{equation}\label{JRblock}
U^* J U = \begin{bmatrix} J_{11}& J_{12}\\ -J_{12}^*& 0 \end{bmatrix},
\quad
U^* R U = \begin{bmatrix} R_{11}& 0\\ 0& 0\end{bmatrix}.
\end{equation}
\end{proposition}

\begin{proof}
Assume (i) and let us take a diagonalisation $U$ of $E$, such  that
$$
U^{*} E U = \begin{bmatrix} E_{11}&0\\ 0&0 \end{bmatrix},\quad  U^{*} AU = \begin{bmatrix} A_{11}& A_{12}\\ A_{21} & A_{22} \end{bmatrix},
$$
with $E_{11}$  invertible. 
For simplicity we can assume from now on that $U=I$. As $\infty$ is an eigenvalue of $\lambda E- A$, the size of $E_{11}$ has to be strictly less than the size of $E$. Observe now that $A_{22}=0$. Indeed, suppose the converse and let $A_{22}x_2\neq0$ for some $x_2\neq 0$.  Consider the vector $x=[0,x_2]^\top$, partitioned conformably with \eqref{EAblock}. Observe that $x\in\ker E$, while there is no vector $y$ satisfying $Ey = Ax$. Hence, in the Kronecker form of $\lambda E-A$ there is a simple Kronecker block corresponding to infinity,  contradiction. 
Finally, since the pencil is regular, $A_{12}$ has full column rank and $A_{21}$ has full row rank, what finishes the proof of the forward implication.

To see the backward implication, notice that a pencil satisfying (ii) is clearly regular.  By Lemma 3 of \cite{MehMW21} it is of index at least 2. Note that there are no simple blocks corresponding to infinity as $A(\ker E)\subseteq \ran E$. 

To show \eqref{JRblock}, observe that $A_{22}=0$ implies $J_{22} = R_{22} = 0$. As $R$ is positive semidefinite, we have $R_{12}=0$ and $R_{21}=0$ as well. 

\end{proof}

%----------------------------------
We prove now the main result.
Observe that the eigenvalue condition number (see, e.g., \cite{stewart1990matrix})  is unsuitable for obtaining non-asymptotic bounds such as those provided in Theorem~\ref{Theorem1}.  The only available perturbation technique 
for tracking the fractional expansion under perturbation appears to be
the eigenvector condition number together with the Bauer-Fike theorem. Although the bound obtained from such reasoning may initially seem pessimistic, in specific examples it provides quite accurate estimates of the range of $\tau$ for which the theorem holds true, see Subsection~\ref{ex:deltas}.

\begin{proof}[Proof of Theorem~\ref{Theorem1}]
First assume that $\delta=0$.  By Proposition 
\ref{prop:index2_form_RJ}
we may assume that the pencil is of the form 
\eqref{EAblock}.
Fix $\tau\in(0,\tau_0)$, and let $\mu=\mu(\tau)$ be as in the definition of $\tau_0$, cf.   \eqref{eq:mu_bound}. We show that it satisfies \eqref{finalsquare2}. Observe that $\eta=\frac1\mu$ is an eigenvalue  of the reversal pencil $\lambda A-(E +\tau I)$, satisfying
\begin{equation}\label{tau0lambda}
\norm{A_{11}}|\eta|<\tau+\sigma_{\min}(E_{11})/2.
\end{equation}
This produces  
\begin{equation}\label{rednew}
\frac12\sigma_{\min}(E_{11} )\norm{x}    \leq \norm{( E_{11}+\tau I-\eta A_{11} )x}\leq \Big(  \frac32 \norm{E_{11}}+2\tau\Big)\norm x,
\quad x\in\Comp^{n_1},
\end{equation}
 where $n_1\times n_1$ is the dimension of $E_{11}$. In particular, the matrix
$E_{11}+\tau I-\eta A_{11}$ is invertible. Hence, the Schur complement of the block matrix $\eta A-(E+\tau I)$,
$$
\tau I - \eta^2 A_{21}(E_{11}+\tau I-\eta A_{11})^{-1}A_{12},
$$
has to be singular. This,  by the Woodbury matrix identity, implies the singularity of 
$$
\eta^2 A_{12}A_{21} - \tau(E_{11}+\tau I-\eta A_{11}).
$$
The latter, in turn, implies that $\mu^2$ is an eigenvalue of the regular linear pencil
$A_{12}A_{21}-\lambda (\tau(E_{11}+\tau I-\eta A_{11}))$, and hence, also of the 
matrix
$$
 A_{12}A_{21}(\tau(E_{11}+\tau I-\eta A_{11}))^{-1}.
$$
Thus,
$$
\frac{\sigma_{\min}(A_{12}A_{21})}{\sigma_{\min}{ (\tau(E_{11}+\tau I-\eta A_{11}))}} \leq  \mu^2 \leq \norm{A_{12}A_{21}} \norm{ (\tau(E_{11}+\tau I-\eta A_{11}))^{-1} },
$$
which, together with \eqref{rednew}, finishes the proof of the inequality \eqref{finalsquare2} in case $\gamma=0$.

Now take  arbitrary $\delta>0$ and let $\delta<\tau<\tau_0$ and $\mu(\tau)$ be as in \eqref{eq:mu_bound}.  Note that $\mu(\tau)$ is also an eigenvalue of the matrix $(E+\tau I)^{-1}A$, as $E+\tau I$ is positive (and hence, nonsingular).  Observe that  
$E+\tau I+\Delta_E$ is nonsingular as well, due to $\delta<\tau$. 
Further, we have
$$
S:=(E+\tau I+\Delta_E)^{-1}(A+\Delta_A)= (E+\tau I)^{-1}A +\Delta,
$$
with some matrix $\Delta$ satisfying a usual estimate
\begin{equation}\label{eq:delta-bounds}
    \norm{\Delta}\leq \delta\cdot  \frac{\tau+\norm A}{\tau(\tau-\delta)},\quad \text{for }\tau>\delta.
\end{equation}

By the Bauer-Fike theorem \cite{bauer1960norms} there exists an eigenvalue $\mu_\Delta (\tau)$ of $S$ 
satisfying
$$
\big| \mu(\tau)-\mu_\Delta(\tau) \big|\leq \gamma(\delta,\tau).
$$
As the eigenvalues of $S$ are also eigenvalues of the perturbed pencil $\lambda ( E +\tau I +\Delta_E)-(A+\Delta_A)$, the proof is finished.

\end{proof}

%---------------------------------------------
\begin{corollary}\label{cor:deltaE0}
Let $E, A, \Delta_E, \Delta_A, \tau_0$ be as in Theorem~\ref{Theorem1}, and assume in addition that $\Delta_E = 0$. 
Then inequality~\eqref{finalsquare2} holds with the sharper error bound on a larger set:
\begin{equation}
\gamma(\delta,\tau)
=
\delta \,\frac{\kappa_V\!\left((E+\tau I)^{-1}A\right)}{\tau},
\qquad 0 < \tau \le \tau_0.
\end{equation}
\end{corollary}

\begin{proof}
It suffices to observe that, in the case $\Delta_E = 0$, the estimate \eqref{eq:delta-bounds} in the proof of Theorem~\ref{Theorem1} improves to
$$
\|\Delta\| \le \frac{\delta}{\tau},\quad \tau>0.
$$
\end{proof}
%---------------------------------------------

\begin{remark}\label{cor:T}
If $E$ is diagonalizable with nonnegative eigenvalues, then the statement of Theorem~\ref{Theorem1} holds with $\norm{X}$, $\sigma_{\min}(X)$, $\kappa_V(X)$ $(X\in\Comp^{n,n})$, and orthogonal projections in formulas \eqref{projections}--\eqref{gamma} understood with respect to the inner product $\langle x,y\rangle_T:=y^*T^{-*}T^{-1}x$, $(x,y\in\Comp^n)$, where $T$ is a diagonalization of $E$.   
Indeed, observe that $E$ is selfadjoint and nonnnegative with respect to this new inner product and so Theorem~\ref{Theorem1} can be applied directly. 
It is informative to clarify the dependencies between the original matrix norm (minimal singular value, condition number, eigenvector condition number, respectively) and the corresponding quantities with respect to the inner product $\langle x,y\rangle_T$, which we denote with the subscript $T$. The following inequalities hold
 \begin{align*}
   &  \frac1{\kappa(T)}\norm{X}\leq \norm{X}_T=\norm{T^{-1}XT}\leq \norm{X}\kappa(T),\\
   & \frac1{\kappa(T)}  \sigma_{\min}(X)  \leq  \sigma_{\min}(X)_{T}=1/\norm{X^{-1}}_T\leq \sigma_{\min}(X)\kappa(T),\\
& \kappa(X)_T\leq\kappa(T)^2\kappa(X),\\
&\kappa_V(X)_T= \inf\{ \kappa(W)_T :\ WXW^{-1} \text{ is diagonal }\}\leq \kappa(T)^2\kappa_V(X).\\
 \end{align*}
%Taking the infimum over $T$, we obtain the stated estimates in terms of $\kappa_V(E)$.
\end{remark}

%---------------------------------------------
\begin{remark}\label{rem:arbitraryE} 
   Let $\lambda E -A\in\Comp^{n,n}[\lambda]$ be as in Theorem~\ref{Theorem1}, with $E$ arbitrary. Let  $U\Sigma V^*$ be a singular value decomposition of $E$.
   To obtain similar bounds for expansions of eigenvalues, we may consider perturbation $\tau UV^*$ instead of $\tau I$. Indeed, 
    multiplying the pencil $\lambda (E+\Delta_E +\tau UV^*)- (A +\Delta_A) $ from the left by $U^*$ and from the right by $V$ we obtain a (strictly equivalent) pencil
    $\lambda( \Sigma + U^*\Delta_EV +\tau I) - (U^*AV+ U^*\Delta_AV)$, to which we then can apply Theorem~\ref{Theorem1}. The details are left to the reader.
\end{remark}

%---------------------------------------------
\begin{remark}\label{rem:finite}
Let $\lambda_0$ be a finite eigenvalue of the pencil $\lambda E - A$, and assume that $A - \lambda_0 E \ge 0$.
To derive results analogous to Theorem~\ref{Theorem1}, we first shift $\lambda_0$ to zero by considering the pencil
$\lambda E - (A  - \lambda_0 E).$
We then apply Theorem~\ref{Theorem1} to the reversal pencil
$\lambda (A  - \lambda_0 E) - E.$
Let $\mu(\tau)$ denote the corresponding eigenvalues of $\lambda (A  - \lambda_0 E + \tau I) - E$.
It follows that the original perturbed pencil $\lambda E - (A + \tau I)$ has eigenvalues of the form
$
\lambda_0 + \frac{1}{\mu(\tau)}
$.
The corresponding estimates are then obtained by performing the same transformation on the inequalities in \eqref{finalsquare2}.
\end{remark}

%%%%%%%%%%%%%%%%%%%%%%%%%%%%%%%%%%%%%%%%%%%%%%%%%%%%%%%%%%%%
\section{Port-Hamiltonian pencils}\label{sec:pH}

As most our examples refer to this class of pencils, we devote the whole section to establish some properties crucial in their analysis. 
We are interested in  matrix pencils of the form
\begin{equation}\label{EJQR}
 \lambda E - (J - R)Q\in\Comp^{n,n}[\lambda],\quad J^* = -J,\  Q^*E  \geq 0,\ R \geq 0,
 \quad 
 \begin{array}{cc}  \text{(a)} & Q \text{  invertible},\\
    \text{(b)} & Q>0\\
   \text{(c)} &  Q=I_n. 
  \end{array}
\end{equation}
Below we discuss the algebraic properties and highlight subtle differences between these choices of $Q$.

%------------------------------------------------------
\begin{theorem} \label{th}
Assume that a pencil $\lambda E_0-A_0$ is regular. Then 
\begin{enumerate}[\rm (i)]
    \item
    $\lambda E_0-A_0$ is of the form \eqref{EJQR}{\rm (a)} 
    if and only if its eigenvalues are contained in the closed left half-plane, finite eigenvalues on the imaginary axis  have Kronecker blocks of size one and the infinity has Kronecker blocks of size at most two.

    \item $\lambda E_0-A_0$ is of the form \eqref{EJQR}{\rm (b)}  if and only if for some invertible $T$ we have
    $\lambda E_0-A_0=T^{-1}( \lambda E - (J-R) ) T$, where
 $\lambda E - (J-R)$ is of the form \eqref{EJQR}{\rm (c)}. 
In particular, in such case $E_0$ is diagonalizable with nonnegative eigenvalues.
\end{enumerate}
\end{theorem}

\begin{proof}
(i) The forward implication follows directly from \cite[Thm.3.1]{Mehl2022}. To see the reverse implication observe that, again by \cite[Thm.3.1]{Mehl2022}, a pencil  $\lambda E_0-A_0$ enjoying  the listed properties satisfies
$$
    \lambda E_0-A_0= S( \lambda  E- ( J-R) )T,\quad E,R\geq 0,\ J=-J^*,
$$
with some invertible $S,T$, and $\lambda E-(J-R)$ of the form \eqref{EJQR}(c). 
Setting $Q:=S^{-*}T$, we obtain
$$
\lambda E_0-A_0= \lambda SET + S(J-R)S^* Q ,\quad Q^* (SET) =T^*ET\geq0. 
$$

(ii) To see the forward implication set $T=Q_0^{-1/2}$, where $Q_0\geq0$ comes from the form \eqref{EJQR}(b)  of the pencil $\lambda E_0-A_0$. Then, for $A_0=J_0-R_0$, we have
$$
T^{-1}(\lambda E_0 - (J_0-R_0)Q_0)T= \lambda Q_0^{1/2} E_0 Q_0^{-1/2} - Q_0^{1/2}(J_0-R_0)Q_0^{1/2},
$$
which is clearly of the form \eqref{EJQR}(c).

The backward implication follows by setting $Q=T^*T$. 
\end{proof}

%------------------------------------------------------
\begin{remark}
\rm Statement (ii) reveals difficulties in characterizing pencils of the form \eqref{EJQR} with $Q>0$ in terms of the Kronecker form. Namely, it can be restated as saying that a given pencil is of the form \eqref{EJQR} with $Q>0$  if and only if it is similar to a pencil with the Kronecker form as in (i), with the similarity matrices $S,T$ satisfying $S=T^{-1}$. Characterizing pencil equivalence with respect to the relation $L(\lambda)=T^{-1}P(\lambda)T$ is not feasible. 
\end{remark}

%------------------------------------------------------
Next we restate Theorem~\ref{Theorem1}  in the port-Hamiltonian setting, obtaining some further simplification. Recall that $\tau_0$ was defined in \eqref{eq:mu_bound}, while $E_{11}$ in \eqref{projections}.

\begin{corollary}
\label{cor:pH}
 Let $\lambda E -(J-R)Q\in\Comp^{n,n}[\lambda]$ be a regular pencil of the form \eqref{EJQR} and of index two with no Kronecker blocks corresponding to infinity of size one. 
\begin{enumerate}[\rm (i)]
    \item 
If $Q=I$ then a perturbed pencil 
\begin{equation}\label{deltaEpH}
    \lambda (E + \tau I  + \Delta_E) -  (A +\Delta_A), 
    \quad \norm{\Delta_E}, \norm{\Delta_A}\leq \delta<\sigma_{\min}(E_{11}),
\end{equation}
 is regular and  has an eigenvalue $\mu_\Delta(\tau)$ satisfying
\begin{equation}\label{finalsquare2pH}
\frac1{\sqrt\tau}{\frac{\sigma_{\min} (J) }{{\sqrt{ \frac 32\norm{E}+2\tau} } }} - \gamma(\delta,\tau)
\leq
|\mu_\Delta(\tau)|
\leq
 \frac1{\sqrt\tau}{\frac{\norm{J} }{\frac 12 \sqrt{\sigma_{\min}( E_{11})}} } +\gamma(\delta,\tau), 
\end{equation}
where
\begin{equation}\label{gammapH}
 \gamma(\delta,\tau)= \delta\cdot \kappa_V( (E +\tau I)^{-1}A  )  \frac{\tau+\norm A}{\tau(\tau-\delta)},\quad 0\leq\delta<\tau<\tau_0.
\end{equation}

Furthermore, in the case $R=0$, the eigenvalue condition number admits the following estimate
\begin{equation}\label{kappaVkappa}
    \kappa_V( (E +\tau I)^{-1}J )\leq\kappa(E+\tau I)\leq \frac{\norm E+\tau}\tau.
\end{equation}

\item For arbitrary $Q>0$, the statements in {\rm (i)} hold true with $\norm{X}$, $\sigma_{\min}(X)$,  $\kappa(X)$, $\kappa_V(X)$ $(X\in\Comp^{n,n})$, and $\tau_0$ understood with respect to the inner product $\langle x, y\rangle_Q =\langle Qx, y\rangle$.

\end{enumerate}
\end{corollary}

\begin{proof}
(i) The first change to prove is that $A_{12}$ in \eqref{finalsquare2} is replaced by $J$ in \eqref{finalsquare2pH}.
This follows from the last part of Proposition~\ref{prop:index2_form_RJ}.

The second addition to Theorem~\ref{Theorem1} in this context  is inequality \eqref{kappaVkappa}. 
 To see this observe that
    the matrix $(E+\tau I)^{-1}J$ is skew-Hermitian with respect to the positive definite inner product $y^*(E+\tau I)x$, $x,y\in\Comp^n$. Hence, its eigenvalue condition number, with respect to the new inner product, equals one. A standard inequality between the induced norms, as in Remark~\ref{cor:T}, finishes the proof. 

(ii)
It is enough to see that $JQ$ is skew-Hermitian and $RQ$ and $E$ are positive semidefinite with respect to the new inner product. 
\end{proof}

\begin{remark}
Recall that the corresponding operator norm and adjoint of  $X\in \Comp^{n,n}$ are given by
\begin{equation*}
   \norm{X}_Q:= \norm{Q^{1/2} X Q^{-1/2}}, \qquad
   X^*_Q = Q^{-*} X^* Q^*.
\end{equation*}
    A dissipative Hamiltonian pencil arising from linearisation (see e.g. Subsection~\ref{ex:Ts} below) shows that it is preferable to use these norms rather than the condition-number-based estimates in Remark~\ref{cor:T}.
\end{remark}

%%%%%%%%%%%%%%%%%%%%%%%%%%%%%%%%%%%%%%%%%%%%%%%%%%%%%%%%%%%%%
\section{Proof of Theorem~\ref{thmM}}\label{sec:thm2}

In \cite{banks2021gaussian}, Banks,  Kulkarni, Mukherjee and Srivastava proved the following major result on estimating eigenvalue condition numbers. We present it here in a form adapted to our notation and purposes.
Recall that if $A$ has distinct eigenvalues $\lambda_1,...,\lambda_n$ and its spectral decomposition is given by $A = \sum^n_{i=1}\lambda_i v_i w_i^* = VDV^{-1}$ where $w_i^* v_i = 1$, then \emph{the eigenvalue condition number} of $\lambda_i$ is defined as 
\begin{equation*}
    \kappa(\lambda_i,A):=\norm{v_i w_i^*}.
\end{equation*}
\begin{theorem} \label{banks}
Let $M \in \mathbb{C}^{n \times n}$ be a matrix with $\|M\| \leq 1$, and let $\tau \in (0,1)$. Consider a complex Ginibre matrix $G_n$. With probability one, the perturbed matrix $M + \tau G_n$ has distinct eigenvalues, which we denote by $\lambda_1(\tau), \ldots, \lambda_n(\tau)$.
Then for every $r>0$ we have
$$
 \mathbb E\left[ \sum_{ |\lambda_i(\tau)|\leq r }\kappa(\lambda_i(\tau), M+\tau G_n)^2
 \right]\leq  \frac{n^2 r^2}{\tau^2}.
$$
\end{theorem}

%-----------------------------------------------
We are ready to prove the second result of the paper.

\begin{proof}[Proof of Theorem~\ref{thmM}]

Fix $\tau>0$ for the whole proof. 
Observe that Theorem \ref{banks} bounds the expected value of a certain part of the sum $\sum_{i=1}^n \kappa(\lambda_i(\tau), M+\tau G_n)^2 $. As the first step of the proof we 
derive  an estimate of the expectation of the whole sum.
For this consider the events 
\begin{align*}
&\Xi_t:=\{t+2\sqrt 2 < \norm{G_n}\le t+1+2\sqrt 2 \},\quad t=0,1,\dots, \\
&\Xi_{-1}:=\{\norm{G_n}\le 2\sqrt 2 \}.
\end{align*}
As $\tau,\norm M\leq 1$, we have
$$
|\lambda_i(\tau)|\leq \norm{M+\tau G_n}\leq  \norm{G_n} +1,\quad i=1,\dots,n. 
$$
Hence,  for $r=t+2+2\sqrt 2$ we have
$$
  \sum_{ i=1}^n\kappa(\lambda_i(\tau), M+\tau G_n)^2 = \sum_{ |\lambda_i(\tau)|\leq r }\kappa(\lambda_i(\tau), M+\tau G_n)^2
 \quad\text{on the event } \Xi_t.
$$
This gives the following bound for the conditional expectation
$$
  \mathbb E\left( \sum_{ i=1}^n\kappa(\lambda_i(\tau), M+\tau G_n)^2\ \Big|\ \Xi_t\right)
 \leq  \frac{n^2 (t+2+2\sqrt 2)^2}{\tau^2}, \quad t=-1,0,1,\dots
$$
By Lemma 2.2 of \cite{banks2021gaussian} we have
$\mathbb P(\Xi_t)\leq 2e^{-nt^2}$, $t=0,1,\dots$ We estimate trivially  $\mathbb P(\Xi_{-1})$ by $1$. Altogether, we obtain

\begin{eqnarray*}
\mathbb E \sum_{ i=1 }^n \kappa(\lambda_i(\tau), M+\tau G_n)^2
&\leq&    \sum_{t=-1}^\infty   \mathbb P(\Xi_t)\frac{n^2(t+2+2\sqrt 2)^2}{\tau^2}    \\
&\leq& \frac{ {(9+4\sqrt{2})} n^2  }{\tau^2} 
+ \frac{2n^2}{\tau^2} \sum_{t=0}^\infty e^{-nt^2} ({t^2 + (4\sqrt{2} + 4) t + 8\sqrt{2} + 12 })   \\
&\leq& \frac{n^2}{\tau^2} 
\left(9+4\sqrt{2} + 24 + 16\sqrt{2} + 2\int\limits_{0}^{\infty} e^{-nt^2} (t^2 + (4\sqrt{2} + 4) t + 8\sqrt{2} + 12) dt\right)\\
&\leq& \frac{n^2}{\tau^2} 
\left(33+20\sqrt{2} + \frac{\sqrt{\pi}}{2n^{3/2}} 
 + \frac{4\sqrt{2} + 4}{n} + \frac{\sqrt{\pi} (8\sqrt{2} + 12) }{n^{1/2}} \right)\\
& =: & \frac{\beta(n)}{\tau^2}  .\label{kV}
\end{eqnarray*}

Recall that the condition number of individual eigenvalues bounds the eigenvalue condition number $\kappa_V$ (see, e.g., Lemma 3.1 of \cite{banks2021gaussian}), hence
\begin{eqnarray}
 \mathbb E(\kappa_V(M+\tau G_n))
 &\leq& \mathbb E \sqrt{ n   \sum_{i=1}^n\kappa(\lambda_i(\tau),M+\tau G_n)^2}\\ 
&\leq &
 \sqrt{n\ \mathbb E \sum_{i=1}^n\kappa(\lambda_i(\tau),M+\tau G_n )^2}\\
&\leq & \frac{\sqrt{n\beta(n)}}{\tau} =: \frac{\alpha(n)}{\tau}.\label{alphaE} 
\end{eqnarray}
Note that, by its definition, $\alpha(n)$ has the following large $n$ asymptotics

$$ \lim_{n\to\infty}\frac{\alpha(n)}{n^{3/2}}\leq \sqrt{ 33+20\sqrt 2 } < 7.82843.
$$

We prove now \eqref{probbound}. Recall that $\tau>0$ is fixed and 
both matrices $M+\tau G_n$ and $M+\Delta+\tau G_n$ are almost surely diagonalizable. Let $\hat G_n$ be a sample of the Ginibre ensemble for which this happens.
Let $\mu_0$ be the eigenvalue of $M+\Delta+\tau \hat G_n$ of minimal absolute value.
By the 
Bauer-Fike theorem \cite{bauer1960norms}, 
there exists an eigenvalue $\lambda_0$ of 
$M+\tau \hat G_n$ with
$|\lambda_0-\mu_0|\leq \norm{\Delta}   \kappa_V(M+\tau \hat G_n)$. Hence,
$$
|\lambda_0|-  \norm{\Delta}   \kappa_V(M+\tau \hat G_n) <   |\mu_0|.
$$
Returning to random variables, we have  almost surely that  
$$
\min_{i=1,\dots,n} |\lambda_i(\tau)| -    \norm{\Delta}   \kappa_V(M+\tau G_n) \leq \min_{i=1,\dots n} |\lambda_{i}^\Delta(\tau)|. 
$$
Employing \eqref{alphaE} we receive
$$
\mathbb E\min_{i=1,\dots,n} |\lambda_i(\tau)| -   \mathbb E\min_{i=1,\dots n} |\lambda_{i}^\Delta(\tau)| \leq \frac{\norm\Delta\alpha(n)}\tau.
$$
By symmetry   of the roles of $M$ and 
$M+\Delta$, the formula \eqref{probbound} is proved. 

\end{proof}

%%%%%%%%%%%%%%%%%%%%%%%%%%%%%%%%%%%%%%%%%%%%%%%%%%%%
\section{Practical implementations }\label{sec:examples}

In this section, we demonstrate the two main theorems of the paper applied to a  collection of benchmark examples commonly studied in the literature. 
The results are collected in Figure~\ref{f:all} for a direct side-by-side comparison. The rows correspond to the individual examples described below, while the columns represent the two proposed methods.

Code and data for reproducing the numerical experiments, including implementations of both methods, are available online at \url{https://github.com/hblazhko/index_detection}. The code also supports custom configurations for additional experiments.

%--------------------------------------------------------
\begin{remark}\label{rem:leftplot} First method of index detection. The plots in the left column of Figure~\ref{f:all} refer to Theorem~\ref{Theorem1}. They show the 
reciprocal $1/|\lambda(\tau)|$ vs. $\tau$ in blue. The parameter $\tau$ ranges from $10^{-20}$ to $10^2$ in all examples.
The reciprocal is computed as the smallest absolute value of the eigenvalues of the reversal pencil, due to numerical reasons. 
The slope of the line should be $1/2$ for index 2 (with no simple eigenvalues infinity) and $1$ for index 1. 

Furthermore, the plots show the reciprocals of the bounds \eqref{finalsquare2}, displayed as orange dash-dot lines (in some cases a different estimate is used, see below for details). Note that for some values of $\tau$ the lower bound in \eqref{finalsquare2} might be negative,  hence trivial. In such cases there is no  upper bound on the left plot. Further, the parameters $\delta$ and $\tau_0$ are marked  with vertical lines (dashed purple and dotted red, respectively). 
\end{remark}

%--------------------------------------------------------
\begin{remark}\label{rem:rightplot} Second method of index detection. 
The plots in the right column  of Figure~\ref{f:all} refer to Theorem~\ref{thmM}. Solid blue lines show the  smallest absolute value of the eigenvalues of 
$$
\frac{(E- hA)^{-1}(E+hA) +I}{\norm{(E- hA)^{-1}(E+hA) +I}} + \tau G_n,
$$
averaged over 10 samples of the complex Ginibre ensemble $G_n$, cf.  Theorem~\ref{thmM}.
The slope of the line should be $1/2$ for index 2 (with no simple eigenvalues infinity) and $1$ for index 1.

The plots additionally show the upper and lower bounds provided by Theorem~\ref{thmM} with $\norm{\Delta} \leq e^{-15}$, indicated by green dash-dot lines.  Observe a visible turning point in the log-log plot, after which the three curves begin to coincide. In particular, to the left of this point the calculated slope of the convergence may be inaccurate due to possible numerical errors. To the right of the turning point, the eigenvalues of the randomized matrix are computed accurately enough to estimate their slope.
\end{remark}

\begin{figure}
    \centering
    \includegraphics[width=\textwidth]{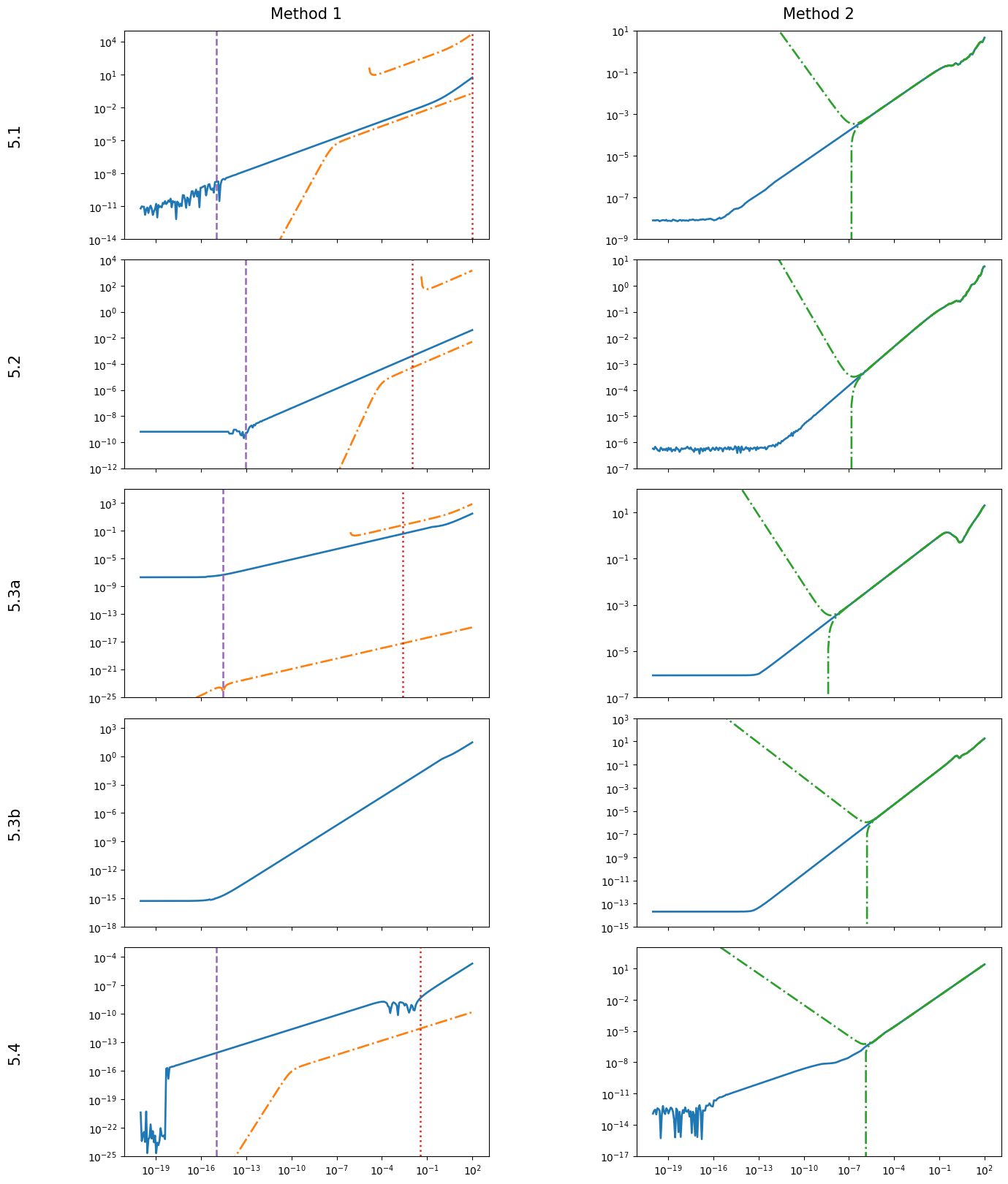}

    \caption{The plots show the reciprocals of eigenvalues corresponding to Theorem~\ref{Theorem1} (left column, solid blue line) and the eigenvalues corresponding to Theorem~\ref{thmM} (right column, solid blue line). The corresponding bounds are shown with dash-dot orange lines. See Remarks~\ref{rem:leftplot} and~\ref{rem:rightplot} for a detailed description. The rows correspond to the subsections, as indicated.
    \label{f:all}}
\end{figure}

%--------------------------------------------------------
\subsection{Toy model}\label{ex:oseen}
   Consider pencil of the form $\lambda E-(J-R)$ with 
\begin{eqnarray}\label{eq:ex1}
E &=  \begin{bmatrix}
I_n & 0   \\
0 & 0
\end{bmatrix}  \in\Comp^{2n,2n},\quad 
A &=  
\begin{bmatrix}
\frac{ R_0  R_0^{*}}{\norm{ R_0}^2}     & -J_0^*   \\
J_0 & 0
\end{bmatrix}.
\end{eqnarray}
Note that for an invertible $J_0$ we have a regular pencil of index 2. 
In our experiments, we sampled $J_0$ and $R_0$ as real $100 \times 100$ matrices with i.i.d. normal entries. 
In particular, both matrices were invertible almost surely.

The numerical results are presented in Figure~\ref{f:all}. 
Both plots clearly exhibit regions of $\tau$ with convergence rate $1/2$, indicating the presence of a Kronecker block of size two. 

The bounds shown for Method~1 are derived from Theorem~\ref{Theorem1} with $\delta = e^{-15}$. 
In this case, the parameter $\tau_0$ from \eqref{eq:mu_bound} exceeds $100$, implying that the bounds are valid throughout the entire displayed range of $\tau$. 
As one can observe, the lower bound closely follows the computed smallest eigenvalue over a certain interval and exhibits the same slope $1/2$. 
The upper bound initially displays similar behavior; however, it loses accuracy earlier and begins to deviate once the lower bound in Theorem~\ref{Theorem1} becomes negative. 

In this model both methods give accurate plots for recognizing the index, which is confirmed by theoretical bounds from Theorem~\ref{Theorem1} and ~\ref{thmM}, respectively.
%-----------------------------------------------------  
\subsection{The influence of a  congruence  transformation} \label{ex:oseen2}

The purpose of this example is to illustrate how the accuracy of the bounds provided in Theorem~\ref{Theorem1} degrades without the knowledge of the explicit block structure \eqref{EAblock}. Let us consider the same pencil \eqref{eq:ex1} from Subsection~\ref{ex:oseen} but transform it with a congruence matrix $S$ sampled as a real $100 \times 100$ matrix with i.i.d. standard normal entries:

\begin{equation}\label{eq:ex1_2}
  S (\lambda E-(J-R))S^*.
\end{equation}

Note that the numerically computed eigenvalues  $s_1\geq\dots\geq s_{200}$ of (positive semidefinite) $SES^*$ are of magnitude $s_{100}\simeq 10^{1}$, while $s_{101} \simeq10^{-13}$.
Therefore, transforming it into the structure \eqref{EAblock}  would require a rank decision.
However, both plots of the eigenvalues shown in Figure~\ref{f:all} do not differ qualitatively from the plot corresponding to the previous case $S=I$ and the index 2 structure is well recognized without the rank decision process. 
The latter is needed only for plotting the bounds from Theorem~\ref{Theorem1}. 
If we interpret
$$
E=U\diag(s_1,\dots,s_{100},0,\dots 0  )U^*,\quad \Delta_E=U\diag(0,\dots 0, s_{101},\dots,s_{200} )U^*,
$$
where $U$ is the eigenvector matrix of $E$,  we have  $\delta=s_{101}\simeq10^{-13}$ and $\tau_0 \approx 0.0106$.

In this model both methods give accurate plots for recognizing the index, which is confirmed by theoretical bounds from Theorem~\ref{Theorem1} and ~\ref{thmM}, respectively.

%----------------------------------------------
\subsection{A dissipative Hamiltonian linearisation of a system of vibrating strings  }\label{ex:Ts}

Consider the quadratic polynomial $
\lambda^2 M + \lambda C + K 
$ arising from the mass-spring-damper system described in \cite[Section 3.9]{TisM01}.
The coefficients have the form
\begin{align*}
&M = \text{diag}(m_1,...,m_n), \\
&C = P\ \text{diag}(d_1, \dots,d_{n-1}, 0)\ P^\top +\diag(t_1,\dots,t_n),\\
&K = P\ \text{diag}(k_1, \dots,k_{n-1}, 0)\ P^\top +\diag(\kappa_1,\dots,\kappa_n),\ \text{with}\\
&P  = [\delta_{ij} - \delta_{i,j+1}]_{i,j=1}^n.
\end{align*}
The corresponding port-Hamiltonian linearization is of the form 
$$
E=\begin{bmatrix} M & 0 \\0&I \end{bmatrix}, \quad  R=\begin{bmatrix} C & 0 \\0&0 \end{bmatrix},\quad 
J=\begin{bmatrix} 0 & -I_n \\I_n&0 \end{bmatrix} \quad 
Q    =\begin{bmatrix} I_n & 0 \\0& K \end{bmatrix}.
$$

%----------------------------------------------
\noindent \textit{Case a)} The theory developed in~\cite{MehMW21} allows us to identify a set of parameters for which the above polynomial lies in a neighborhood of index~2 polynomials, while remaining relatively far from singular polynomials. 
Specifically, we can set $m_1 = d_1 = t_1 = \varepsilon \ll 1$, while keeping the remaining coefficients separated from zero. By  \cite{MehMW21} we obtain in this way a pencil close to index 2, however, in a safe distance from the set of singular pencils. 

Figure~\ref{f:all}, row 5.3a, presents both methods applied to the case of $n=10$ strings with 
$\varepsilon = e^{-15}$. The bounds shown for Method~1 are obtained from Corollary~\ref{cor:pH}(ii) for port-Hamiltonian pencils with $Q>0$, where 
$\delta$ is taken equal to $3\varepsilon =3 e^{-15}$. In this case, $\tau_0$ from \eqref{eq:mu_bound} equals approximately $0.002$.

%----------------------------------------------
\noindent\textit{Case b)} 
On the other hand, setting $m_1 = \varepsilon \ll 1$ while keeping $d_1$ and $t_1$ bounded away from zero leads to pencil close to index 1,  in a safe distance from the set of singular pencils and pencils of index 2. 

Figure~\ref{f:all}, row 5.3b, shows both methods applied to the case of $n=10$ strings with $\varepsilon = e^{-15}$. 
Since the resulting pencil is of index~1, the bounds provided for Method~1 are not applicable in this setting. 
Nevertheless, both plots clearly exhibit a convergence rate equal to~1.

In this model both methods give accurate plots for recognizing the index, confirmed by Theorems~\ref{Theorem1}, and Theorem~\ref{thmM}, respectively. 
Thus, a clear distinction between pencils near to index 1 and near to index 2 is visible, even though the perturbed pencil (for which the eigenvalues were computed) has no eigenvalue at infinity.

%----------------------------------------------
\subsection{Electrical circuit}\label{ex:CMOS} 

Consider the two-stage CMOS operational amplifier introduced in \cite{sommer2012application} and later studied from the numerical point of view in \cite{GERNANDT2025429}. The associated matrix pencil has an infinite eigenvalue of algebraic multiplicity 6 and is of index 2. 

The numerical problem arising in this example is that the nonzero entries of $E$ and $A$ are of magnitude $10^{-14}$ and $10^{-4}$, respectively, which is too small to obtain the desired convergence for $\tau$ ranging from $10^{-20}$ to $10^2$. To keep $\tau$ in the same range in all examples we multiply both matrices by $e^{10}$. 
Both plots presented in Figure~\ref{f:all} show convergence rate $1/2$ for a certain range of $\tau$.
Furthermore, $\tau_0$ marked on the plot corresponds in this situation to the place where the smallest eigenvalue (of the reversal pencil) arising from the origin, interferes with other eigenvalues, as visible on the plot. 

In this example both Methods produce a clear recognition of index 2. Further, Method 1 is confirmed by theoretical bounds. 
Moreover, in this example the theoretical bounds from Theorem~\ref{thmM} lose their accuracy  due to the different magnitudes of $A$ and $E$.

%----------------------------------------------
\subsection{Errors in the staircase form}\label{ex:deltas}

In practice, the matrix $E$ is typically diagonal, and hence its kernel is clearly identified. 
The difficulty in determining whether a system is of index~2 lies in verifying whether the matrix $A_{22}$ in the block form~\eqref{EAblock} is a zero matrix. 
To highlight this, we examine how perturbations in the matrix $A$ affect the behavior of the eigenvalues and the bounds in Theorem~\ref{Theorem1}. 
To this end, we consider the pencil~\eqref{eq:ex1} from Subsection~\ref{ex:oseen}, perturbed by a small matrix $\Delta_A$.

Figure~\ref{f:deltas} shows Method~1 applied to the pencil perturbed by a matrix $\Delta_A$ sampled as a 
$200 \times 200$ matrix with  i.i.d.\ real Gaussian entries of variance $e^{-7}$, $e^{-5}$, and $e^{-3}$ (from left to right, respectively). 
The displayed bounds are obtained from Corollary~\ref{cor:deltaE0} for the case $\Delta_E = 0$, with $\delta = e^{-7}, e^{-5}, e^{-3}$, respectively. 

Note that the estimates in Corollary~\ref{cor:deltaE0} hold for $\tau>0$,  not only $\tau>\delta$, which was previously the case. Hence, the value of $\delta$ is not marked on the plot. 
In all cases $\tau_0$ exceeds 100 so the obtained bounds are meaningful on the whole considered range of $\tau$, thus it is not marked as well.

Although the pencil~\eqref{eq:ex1} is of index~2, larger perturbation magnitudes lead to wider regions in which the convergence rate equals~1, corresponding to pencils of index~1. 
In particular, this correspondence allows one to estimate the distance to the set of index~2 pencils via the observed convergence behavior. 
We also note that, in this example, the provided bounds are nearly sharp for $\delta=e^{-3}$.

One can observe that the range of values of $\tau$ can be divided into four distinct regions. 

\begin{enumerate}[\rm (I)]
\item Too large values of $\tau$ (relative to the scale of $E, A$), where the statement is not yet valid and its effect cannot be observed;
\item Optimal values, where the convergence rate given in the Puiseux expansion \eqref{eq:expansion} 
can be observed;
\item Suboptimal values of $\tau$, i.e., values for which the numerics change a Kronecker block of size 2 into two blocks of size 1;
\item Too small values of $\tau$, for which round-off error effects hinder any computation.
\end{enumerate}

%---------------------------------------------
\begin{figure}
    \centering
    \includegraphics[width=\textwidth]{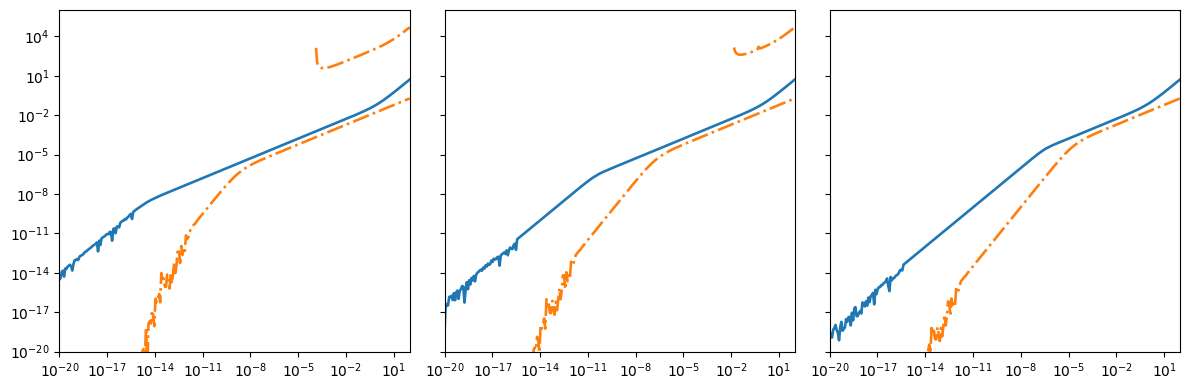}

    \caption{ The plots show the reciprocals of eigenvalues corresponding to Theorem~\ref{Theorem1} with solid blue lines. Corresponding bounds are shown with dash-dot orange lines. See Remark~\ref{rem:leftplot} for a full description. The plots correspond to the variance $e^{-7}$, $e^{-5}$, $e^{-3}$  of the entries of $\Delta_A$ from Subsection~\ref{ex:deltas}, respectively.}\label{f:deltas}
\end{figure}

%%%%%%%%%%%%%%%%%%%%%%%%%%%%%%%%%%%%%%%%%%%%%%%%%%%%%%%%%%%%%%%%%
\section{Conclusions}
We have answered questions (Q1) and (Q2) positively. 
We have demonstrated, both theoretically and numerically, that the second order Kronecker blocks are visible and detectable in the expansion of eigenvalues. Furthermore,  
the Cayley transform preserves this feature despite the numerical error. 
We have provided heuristic evidence adressing question (Q3).

\section{Acknowledgments}

Hanna Blazhko gratefully acknowledges the support of the program Excellence Initiative – Research University at the Jagiellonian University in Kraków.

\bibliographystyle{plain}
\bibliography{bibi}

@article{akinola2014calculation,
  title={The calculation of the distance to a nearby defective matrix},
  author={Akinola, Richard O and Freitag, Melina A and Spence, Alastair},
  journal={Numerical Linear Algebra with Applications},
  volume={21},
  number={3},
  pages={403--414},
  year={2014},
  publisher={Wiley Online Library}
}

@article{erbay2024index,
  title={Index concepts for linear differential-algebraic equations in infinite dimensions},
  author={Erbay, Mehmet and Jacob, Birgit and Morris, Kirsten and Reis, Timo and Tischendorf, Caren},
  journal={DAE Panel},
  volume={2},
  year={2024}
}

@article{VanDooren1979,
  author  = {Van Dooren, Paul},
  title   = {The computation of {Kronecker}'s canonical form of a singular pencil},
  journal = {Linear Algebra and its Applications},
  volume  = {27},
  year    = {1979},
  pages   = {103--140},
  doi     = {10.1016/0024-3795(79)90035-1}
}

@incollection{gillis2022solving,
  title={Solving Matrix Nearness Problems via {H}amiltonian Systems, Matrix Factorization, and Optimization},
  author={Gillis, Nicolas and Sharma, Punit},
  booktitle={Recent Stability Issues for Linear Dynamical Systems: Cetraro, Italy 2021},
  pages={1--83},
  year={2024},
  publisher={Springer}
}

@article{gillis2018finding,
  title={Finding the nearest positive-real system},
  author={Gillis, Nicolas and Sharma, Punit},
  journal={SIAM Journal on Numerical Analysis},
  volume={56},
  number={2},
  pages={1022--1047},
  year={2018},
  publisher={SIAM}
}

@article{alam2005construction,
  title={On the construction of nearest defective matrices to a normal matrix},
  author={Alam, Rafikul},
  journal={Linear algebra and its applications},
  volume={395},
  pages={367--370},
  year={2005},
  publisher={Elsevier}
}

@incollection{kressner2015distance,
  title={Distance problems for linear dynamical systems},
  author={Kressner, Daniel and Voigt, Matthias},
  booktitle={Numerical algebra, matrix theory, differential-algebraic equations and control theory: {F}estschrift in honor of {V}olker {M}ehrmann},
  pages={559--583},
  year={2015},
  publisher={Springer}
}

@inproceedings{giesbrecht2017computing,
  title={Computing the nearest rank-deficient matrix polynomial},
  author={Giesbrecht, Mark and Haraldson, Joseph and Labahn, George},
  booktitle={Proceedings of the 2017 ACM International Symposium on Symbolic and Algebraic Computation},
  pages={181--188},
  year={2017}
}

@article{faulwasser2022optimal,
  title={Optimal control of port-{H}amiltonian descriptor systems with minimal energy supply},
  author={Faulwasser, Timm and Maschke, Bernhard and Philipp, Friedrich and Schaller, Manuel and Worthmann, Karl},
  journal={SIAM Journal on Control and Optimization},
  volume={60},
  number={4},
  pages={2132--2158},
  year={2022},
  publisher={SIAM}
}

@article{kotyczka2021symplectic,
  title={Symplectic discrete-time energy-based control for nonlinear mechanical systems},
  author={Kotyczka, Paul and Thoma, Tobias},
  journal={Automatica},
  volume={133},
  pages={109842},
  year={2021},
  publisher={Elsevier}
}

@article{mehrmann2023control,
  title={Control of port-Hamiltonian differential-algebraic systems and applications},
  author={Mehrmann, Volker and Unger, Benjamin},
  journal={Acta Numerica},
  volume={32},
  pages={395--515},
  year={2023},
  publisher={Cambridge University Press}
}

@article{bauer1960norms,
  title={Norms and exclusion theorems},
  author={Bauer, Friedrich L and Fike, Charles T},
  journal={Numerische mathematik},
  volume={2},
  number={1},
  pages={137--141},
  year={1960},
  publisher={Springer}
}

@article{saak2018model,
  title={Model reduction of constrained mechanical systems in {MM. ESS}},
  author={Saak, Jens and Voigt, Matthias},
  journal={IFAC-PapersOnLine},
  volume={51},
  number={2},
  pages={661--666},
  year={2018},
  publisher={Elsevier}
}

@article{emmrich2013operator,
  title={Operator differential-algebraic equations arising in fluid dynamics},
  author={Emmrich, Etienne and Mehrmann, Volker},
  journal={Computational Methods in Applied Mathematics},
  volume={13},
  number={4},
  pages={443--470},
  year={2013},
  publisher={Walter de Gruyter GmbH}
}

@article{alam2011characterization,
  title={Characterization and construction of the nearest defective matrix via coalescence of pseudospectral components},
  author={Alam, Rafikul and Bora, Shreemayee and Byers, Ralph and Overton, Michael L},
  journal={Linear Algebra and its Applications},
  volume={435},
  number={3},
  pages={494--513},
  year={2011},
  publisher={Elsevier}
}

@article{alam2005sensitivity,
  title={On sensitivity of eigenvalues and eigendecompositions of matrices},
  author={Alam, R and Bora, S},
  journal={Linear Algebra and its applications},
  volume={396},
  pages={273--301},
  year={2005},
  publisher={Elsevier}
}

@article{sun1991perturbation,
  title={Perturbation expansions for invariant subspaces},
  author={Sun, Ji-guang},
  journal={Linear Algebra and its Applications},
  volume={153},
  pages={85--97},
  year={1991},
  publisher={Elsevier}
}

@article{kressner2023singular,
  title={Singular quadratic eigenvalue problems: Linearization and weak condition numbers},
  author={Kressner, Daniel and {\v{S}}ain Glibi{\'c}, Ivana},
  journal={BIT Numerical mathematics},
  volume={63},
  number={1},
  pages={18},
  year={2023},
  publisher={Springer}
}

@article{stewart1972sensitivity,
  title={On the sensitivity of the eigenvalue problem ${A}x=\lambda {B}x$},
  author={Stewart, Gilbert W},
  journal={SIAM Journal on Numerical Analysis},
  volume={9},
  number={4},
  pages={669--686},
  year={1972},
  publisher={SIAM}
}

@article{ahmad2010pseudospectra,
  title={On pseudospectra, critical points, and multiple eigenvalues of matrix pencils},
  author={Ahmad, Sk Safique and Alam, Rafikul and Byers, Ralph},
  journal={SIAM Journal on Matrix Analysis and Applications},
  volume={31},
  number={4},
  pages={1915--1933},
  year={2010},
  publisher={SIAM}
}

@book{stewart1990matrix,
  title={Matrix perturbation theory},
  author={Stewart, Gilbert W and Sun, Ji-guang},
address = {new York},
  year={1990},
publisher = {Academic Press},
}

@article{freund2004extension,
  title={An extension of the positive real lemma to descriptor systems},
  author={Freund, Roland W and Jarre, Florian},
  journal={Optimization methods and software},
  volume={19},
  number={1},
  pages={69--87},
  year={2004},
  publisher={Taylor \& Francis}
}

@article{gillis2018computing,
  title={Computing the nearest stable matrix pairs},
  author={Gillis, Nicolas and Mehrmann, Volker and Sharma, Punit},
  journal={Numerical Linear Algebra with Applications},
  volume={25},
  number={5},
  pages={e2153},
  year={2018},
  publisher={Wiley Online Library}
}

@article{Kressner_2024,
   title={Analysis of eigenvalue condition numbers for a class of randomized numerical methods for singular matrix pencils},
   volume={64},
   ISSN={1572-9125},
   url={http://dx.doi.org/10.1007/s10543-024-01033-w},
   DOI={10.1007/s10543-024-01033-w},
   number={3},
   journal={BIT Numerical Mathematics},
   publisher={Springer Science and Business Media LLC},
   author={Kressner, Daniel and Plestenjak, Bor},
   year={2024},
   }

@article{Hochstenbach_2023,
author = {Hochstenbach, Michiel E. and Mehl, Christian and Plestenjak, Bor},
title = {Solving Singular Generalized Eigenvalue Problems. {P}art {I}{I}: Projection and Augmentation},
journal = {SIAM Journal on Matrix Analysis and Applications},
volume = {44},
number = {4},
pages = {1589-1618},
year = {2023},
doi = {10.1137/22M1513174},

URL = {  https://doi.org/10.1137/22M1513174},
eprint = {  https://doi.org/10.1137/22M1513174}
}

@article{Hochstenbach_2019,
   title={Solving Singular Generalized Eigenvalue Problems by a Rank-Completing Perturbation},
    author={Hochstenbach, Michiel E and Mehl, Christian and Plestenjak, Bor},
  journal={SIAM Journal on Matrix Analysis and Applications},
  volume={40},
  number={3},
  pages={1022--1046},
  year={2019},
  publisher={SIAM}
}

@article{TisM01,
  author    = {F. Tisseur and K. Meerbergen},
  title     = {The Quadratic Eigenvalue Problem},
  journal   = {{SIAM} Review},
  volume    = {43},
  pages     = {235--286},
  year      = {2001}
}

@incollection{sommer2012application,
  title={Application of symbolic circuit analysis for failure detection and optimization of industrial integrated circuits},
  author={Sommer, Ralf and Krau{\ss}e, Dominik and Sch{\"a}fer, Eric and Hennig, Eckhard},
  booktitle={Design of Analog Circuits through Symbolic Analysis},
  pages={445--477},
  year={2012},
  publisher={Bentham Science Publishers}
}

@article{Lidski66,
  author    = {V.~B.~Lidskii},
  title     = {Perturbation theory of non-conjugate operators},
  journal   = {Zh. Vychisl. Mat. Mat. Fiz.},
  year      = {1966},
  volume    = {6},
  number    = {1},
  pages     = {52--60},
  note      = {Translated in \emph{U.S.S.R. Comput. Math. Math. Phys.}, \textbf{6}(1):73--85, 1966},
  url       = {http://mi.mathnet.ru/zvmmf7501},
  mrnumber  = {0196930},
  zbl       = {0166.40501},
  doi       = {10.1016/0041-5553(66)90033-4}
}

@article{Wilkinson1979,
  author  = {J. H. Wilkinson},
  title   = {Kronecker's Canonical Form and the {QZ} Algorithm},
  journal = {Linear Algebra and its Applications},
  volume  = {28},
  pages   = {285--303},
  year    = {1979}
}

@article{MehMW18,
  title={Linear algebra properties of dissipative {H}amiltonian descriptor systems},
  author={Mehl, C. and Mehrmann, V. and Wojtylak, M.},
  journal={{S}{I}{A}{M} Journal on Matrix Analysis and Applications},
  volume={39},
  number={3},
  pages={1489--1519},
  year={2018},
  publisher={SIAM}
}

@article{MehMW21,
	author = {{Mehl}, C. and {Mehrmann}, V. and {Wojtylak}, M. },
	title = {Distance problems for dissipative {H}amiltonian systems and related matrix polynomials},
	journal = {Linear algebra and its applications},
	year = {2021},
	volume = {623},
	pages = {335--366},
}

@article{de2008first,
  title={First order spectral perturbation theory of square singular matrix pencils},
  author={De Ter{\'a}n, Fernando and Dopico, Froil{\'a}n M and Moro, Julio},
  journal={Linear algebra and its applications},
  volume={429},
  number={2-3},
  pages={548--576},
  year={2008},
  publisher={Elsevier}
}

@article{mackey2015mobius,
  title={M{\"o}bius transformations of matrix polynomials},
  author={Mackey, D Steven and Mackey, Niloufer and Mehl, Christian and Mehrmann, Volker},
  journal={Linear Algebra and its Applications},
  volume={470},
  pages={120--184},
  year={2015},
  publisher={Elsevier}
}

@article{Mehl2022,
  author    = {Christian Mehl and Volker Mehrmann and Michał Wojtylak},
  title     = {Matrix Pencils with Coefficients that have Positive Semidefinite {H}ermitian Parts},
  journal   = {SIAM Journal on Matrix Analysis and Applications},
  volume    = {43},
  year      = {2022},
  pages     = {1186--1212}
}

@article{banks2021gaussian,
  title={Gaussian regularization of the pseudospectrum and Davies’ conjecture},
  author={Banks, Jess and Kulkarni, Archit and Mukherjee, Satyaki and Srivastava, Nikhil},
  journal={Communications on Pure and Applied Mathematics},
  volume={74},
  number={10},
  pages={2114--2131},
  year={2021},
  publisher={Wiley Online Library}
}

@article{GERNANDT2025429,
title = {Eigenvalues of parametric rank-one perturbations of matrix pencils},
journal = {Linear Algebra and its Applications},
volume = {708},
pages = {429-457},
year = {2025},
issn = {0024-3795},
doi = {https://doi.org/10.1016/j.laa.2024.12.012},
url = {https://www.sciencedirect.com/science/article/pii/S0024379524004841},
author = {Hannes Gernandt and Carsten Trunk}
}

\end{document}